\documentclass[12pt]{amsart}
\usepackage{amssymb}
\usepackage{amsthm}
\usepackage{amsmath}
\usepackage{enumerate}
\usepackage[dvips]{graphics,color}
\usepackage{color}
\usepackage{pgf,tikz}

\theoremstyle{plain}
\newtheorem{Thm}{Theorem}[section]
\newtheorem{Lem}[Thm]{Lemma}

\newtheorem{Cor}[Thm]{Corollary}

\theoremstyle{remark}

\small\normalsize
\numberwithin{equation}{section}

\newcommand{\T}{{\mathbb T}}
\newcommand{\C}{{\mathbb C}}

\newcommand{\R}{{\mathbb R}}
\newcommand{\D}{{\mathbb D}}

\begin{document}

\title{On a characterization of finite Blaschke products}

\address{Universit\'e de Lyon; Universit\'e Lyon 1; Institut Camille Jordan CNRS UMR 5208; 43, boulevard du 11 Novembre 1918, F-69622 Villeurbanne.}
\email{fricain@math.univ-lyon1.fr}

\author{Emmanuel Fricain, Javad Mashreghi}
\address{D\'epartement de math\'ematiques et de statistique,
         Universit\'e Laval,
         Qu\'ebec, QC,
         Canada G1K 7P4.}
\email{Javad.Mashreghi@mat.ulaval.ca}

\thanks{This work was supported by NSERC (Canada), Jacques Cartier Center and ANR project FRAB (France).}

\keywords{Blaschke products, zero sets, convergence, automorphism}

\subjclass[2000]{Primary: 30D50, Secondary: 32A70}

\begin{abstract}
We study the convergence of a sequence of finite Blaschke products of a fix order toward a rotation. 
This would enable us to get a better picture of a characterization theorem for finite Blaschke products.
\end{abstract}

\maketitle

\section{Introduction}

Let $(z_k)_{1 \leq k \leq n}$ be a finite sequence in the open unit disc $\D$. Then the rational function
\[
B(z) = \gamma \, \prod_{k=1}^{n} \frac{z_k-z}{1-\bar{z}_k \, z},
\]
where $\gamma$ is a unimodular constant, is called a finite Blaschke product of order $n$ for $\D$ \cite{pD70}. There are various results characterizing these functions. For example, one of the oldest ones is due to Fatou.\\

\noindent {\bf Theorem A} (Fatou \cite{MR1504825}){\bf .}
{\em Let $f$ be analytic in the open unit disc $\D$ and suppose that
\[
\lim_{|z| \to 1} |f(z)| = 1.
\]
Then $f$ is a finite Blaschke
product.}\\

For an analytic function $f : \Omega_1 \longrightarrow \Omega_2$, the number of solutions of the equation
\[
f(z) = w, \hspace{1cm} (z \in \Omega_1, \, w \in \Omega_2),
\]
counting multiplicities, is called the {\em valence} of $f$ at $w$ and is denoted by $v_f(w)$. It is well-known that a finite Blaschke product of order $n$ has the constant valence $n$ for each $w \in \D$. But, this property in fact characterizes finite Blaschke products of order $n$.\\

\noindent {\bf Theorem B} (Fatou \cite{MR1504787, MR1504792, MR1504797}, Rad\'o \cite{rad0-1922}){\bf .}
{\em Let $f : \D \longrightarrow \D$ be an analytic function of constant valence $n \geq 1$. Then $f$ is a finite Blaschke products of order $n$.}\\

Our main concern in this paper is the following result of Heins.  We remind that a conformal automorphism of $\D$ has the general form
\[
T_{a,\gamma}(z) =   \gamma \,  \frac{a-z}{1-\bar{a} \, z},
\]
where $a \in \D$ and $\gamma \in \T$. Instead of $T_{a,1}$ we simply write $T_{a}$. The collection of all such elements is denoted by $\mbox{Aut}(\D)$.
The special automorphism
\[
\begin{array}{cccc}
\rho_\gamma: & \D & \longrightarrow & \D \\
& z & \longmapsto & \gamma z
\end{array}
\]
is called a rotation. Note that $\rho_\gamma = T_{0,-\gamma}$.

Let us remind two further notions. In the following, when we say that a sequence of functions on $\D$ is convergent, we mean that it converges uniformly on compact subsets of $\D$. If $f : \D \longrightarrow \C$ and $(a_k)_{k \geq 1} \subset \D$, with $\lim_{k \to \infty} |a_k|=1$, is such that $\lim_{k \to \infty} f(a_k) = L$, then $L$ is called an asymptotic value of $f$.\\

\noindent {\bf Theorem C} (Heins \cite{heins86}){\bf .}
{\em Let $f : \D \longrightarrow \D$ be an analytic function. Then the following assertions are equivalent:
\begin{enumerate}[(i)]
\item $f$ is a finite Blaschke product of order $\geq 1$;
\item if the sequence of automorphisms $T_{a_k,\gamma_k}$, $k \geq 1$, tends to a constant of modulus one, and if $g_k = T_{f(a_k\gamma_k)} \circ f \circ T_{a_k,\gamma_k}$ is convergent, then $g_k$ tends to a rotation;
\item we have
\[
\lim_{|z| \to 1} \frac{(1-|z|^2) \, \, |f'(z)|}{1-|f(z)|^2} = 1;
\]
\item $f$ has no asymptotic values in $\D$ and the set $\{z \in \D : f'(z) = 0 \}$ is finite.\\
\end{enumerate}}

Our contribution is to further clarify the item $(ii)$ above in which it is assumed that the limits of
two sequences of analytic functions exist. No doubt the existence of these limits depend on the parameters
$a_k \in \D$ and $\gamma_k \in \T$, $k \geq 1$.

It is easy to see that $(T_{a_k,\gamma_k})_{k \geq 1}$ tends to a unimodular constant $\gamma_0$ if and only if
\begin{equation} \label{E:t-akgammaktogammao}
\lim_{k \to \infty} a_k \, \gamma_k = \gamma_0.
\end{equation}
In the first place, we show that if $B$ is a finite Blaschke product of order $\geq 1$ and (\ref{E:t-akgammaktogammao}) holds, then
\[
\big(\, T_{B(a_k\gamma_k),\, \bar{\gamma}_k} \circ B \circ T_{a_k,\gamma_k} \,\big) (z) \,\, \longrightarrow \,\, \frac{B'(\gamma_0)}{|B'(\gamma_0)|} \, z
\]
as $k \longrightarrow +\infty$ and the convergence is uniform on compact subsets of $\D$. See Theorem~\ref{T:conv-sak}. Therefore, the sequence $T_{B(a_k\gamma_k)} \circ B \circ T_{a_k,\gamma_k}$,
which was considered in Theorem C, is convergent to a rotation if and only if $\gamma_k$ is convergent. We refer
to the example given at the end to see that this condition cannot be relaxed.
In this special situation, if $a_k \longrightarrow \gamma_0$, as $k \longrightarrow +\infty$, then
\[
\big(\, T_{B(a_k)} \circ B \circ T_{a_k} \,\big) (z) \,\, \longrightarrow \,\, \frac{B'(\gamma_0)}{|B'(\gamma_0)|} \, z,
\]
uniformly on compact subsets of $\D$. The above observations enable us to
rewrite Theorem C in the following form.\\

\noindent {\bf Theorem D.}
{\em Let $f : \D \longrightarrow \D$ be an analytic function. Then the following assertions are equivalent:
\begin{enumerate}[(i)]
\item $f$ is a finite Blaschke product of order $\geq 1$;
\item if $\gamma_k \, a_k \longrightarrow \gamma_0 \in \T$,
 then $T_{f(a_k\gamma_k), \bar{\gamma}_k} \circ f \circ T_{a_k,\,\gamma_k}$ tends to a rotation;
\item if $a_k \longrightarrow \gamma_0 \in \T$,
 then $T_{f(a_k)} \circ f \circ T_{a_k}$ tends to a rotation;
\item the equality
\[
\lim_{|z| \to 1} \frac{(1-|z|^2) \, \, |f'(z)|}{1-|f(z)|^2} = 1
\]
holds;
\item $f$ has no asymptotic values in $\D$ and the set $\{z \in \D : f'(z) = 0 \}$ is finite;
\item $f$ has a constant valence on $\D$.
\end{enumerate}
Moreover, if the above conditions hold, then the rotation promised in $(ii)$ and $(iii)$ is $\rho_\gamma$, where $\gamma = f'(\gamma_0)/|f'(\gamma_0)|$.\\}

We will just study the implication $(i) \Longrightarrow (ii)$ in Theorem~\ref{T:conv-sak}. That $(ii) \Longrightarrow (iii)$ is trivial. The implication $(iii) \Longrightarrow (iv)$ is an easy consequence of the formula
\[
\frac{(1-|a_k|^2) \, \, f'(a_k)}{1-|f(a_k)|^2} = \big(\, T_{f(a_k)} \circ f \circ T_{a_k} \,\big)' (0),
\]
and that, by assumption, the latter tends to a unimodular constant. The more delicate steps $(iv) \Longrightarrow (v) \Longrightarrow (vi)$ and $(vi) \Longrightarrow (i)$ are already taken respectively by Heins and Fatou--Rad\'o.

In the course of proof, we also show that, for a finite Blaschke product $B$, the zeros of $B'$ which are inside $\D$ are in the hyperbolic convex hull of the zeros of $B$.

\section{hyperbolic convex hull}

The relation $B(z) \, \overline{B(1/\bar{z})} = 1$ shows that
\[
B'(z) = 0 \Longleftrightarrow B'(1/\bar{z}) = 0.
\]
Hence, to study the singular points of $B$, it is enough to consider the zeros of $B'$ which are inside
the unit disc $\D$. Note that $B$ has no singular points on $\T$, according to
\eqref{eq:derive-au-bord} for instance. Moreover if $B$ has $n$ zero in $\D$, then it is easy to see that
$B'$ has $n-1$ zeros in $\D$ (counting with multiplicities).

There are various results about the relations between the zeros of a polynomial $P$ and the zeros of its
derivatives. The eldest goes back to Gauss--Lucas \cite{luc74} which says that the zeros of $P'$ are
in the convex hull of the zeros of $P$. Recently, Cassier--Chalendar \cite{cc20} established a similar
result for Blaschke products: the zeros of $B'$ in $\D$ are in the convex hull of the zeros of $B$ and $\{0\}$.
We show that the zeros of $B'$ are in the hyperbolic convex hull of the zeros of $B$.
We do not need to enlarge the zero sets of $B$ by adding $\{0\}$. Moreover, the hyperbolic convex hull of a set
is a subset of the Euclidean convex hull of the set and $\{0\}$. In particular, we improve the result
obtained  in \cite{cc20}.

Let $z_1,z_2 \in \D$. Then the hyperbolic line between $z_1$ and $z_2$ is given by
\[
\begin{array}{ccc}
[0,1] & \longrightarrow & \D\\
t & \longmapsto & \displaystyle \frac{z_1 - \frac{z_1-z_2}{1-\overline{z}_1 z_2} \, t}{1-\overline{z}_1 \, \frac{z_1-z_2}{1-\overline{z}_1 z_2} \, t}.
\end{array}
\]
This representation immediately implies that the three distinct points $z_1,z_2,z_3 \in \D$ are on the same hyperbolic line if and only if
\begin{equation} \label{E:z1z2z3online}
\left( \frac{z_1-z_2}{1-\overline{z}_1 z_2} \right) / \left( \frac{z_1-z_3}{1-\overline{z}_1 z_3} \right) \in \R.
\end{equation}
Furthermore, we see that the hyperbolic lines in $\D$ can be parameterized by
\[
\gamma \frac{a-z}{1-\overline{a} z}=t,\qquad t\in [-1,1],
\]
with $\gamma\in\T$ and $a\in\D$.

Adopting the classical definition from the Euclidean geometry, we say that a set $A \subset \D$ is hyperbolically convex if \[
z_1,z_2 \in A \Longrightarrow \forall t \in [0,1], \,\, \displaystyle \frac{z_1 - \frac{z_1-z_2}{1-\overline{z}_1 z_2} \, t}{1-\overline{z}_1 \, \frac{z_1-z_2}{1-\overline{z}_1 z_2} \, t} \in A.
\]
The hyperbolic convex hull of a set $A$ is the smallest hyperbolic convex set which contains $A$. This is clearly the intersection of all hyperbolic convex sets that contain $A$. Of course, if we need the closed hyperbolic convex hull, we must consider the intersection of all closed hyperbolic convex sets that contain $A$.

In the proof of the following theorem, we need at least two elementary properties of the automorphism
\[
T_a(z) = \frac{a-z}{1-\bar{a} \, z}, \hspace{1cm} (a,z \in
\D).
\]
Firstly, $T_a \circ T_a = id$. Secondly, if $a \in (-1,1)$, then $T_a$ maps $\D_{-} = \D \cap \{ z : \Im z < 0\}$ into $\D_+ = \D \cap \{ z : \Im z > 0\}$. In the same token, we define
$\overline{\D}_{+} = \D \cap \{ z : \Im z \geq 0\}$.

\begin{Thm} \label{T:hyper-convex}
Let $B$ be a finite Blaschke product. Then the zeros of $B'$ which are inside $\D$ are in the hyperbolic convex hull of the zeros of $B$.
\end{Thm}

\begin{proof}
First suppose that all zeros of $B$ are in $\D_{+}$. Then, taking the logarithmic derivative of $B$, we obtain
\[
\frac{B'(z)}{B(z)} = \sum_{k=1}^{n}
\frac{1-|z_k|^2}{(1-\bar{z}_k \, z)(z-z_k)},
\]
which implies
\[
\Im \left( \frac{B'(z)}{B(z)} \right) = \sum_{k=1}^{n}
\Im \left( \frac{1-|z_k|^2}{(1-\bar{z}_k \, z)(z-z_k)} \right).
\]
In the light of last expression, put
\[
\Phi(z) = \frac{1-|a|^2}{(1-\bar{a} \, z)(z-a)},
\]
where $a \in \D_+$ is fixed. We need to obtain the image of $\D_{-}$ under $\Phi$. To do so, we study the image of the boundary of $\partial \D_{-}$ under $\Phi$. On the lower semicircle
\[
\T_{-} = \{ e^{i\theta} : -\pi \leq \theta \leq 0 \}
\]
we have
\[
\Phi(e^{i\theta}) = \frac{1-|a|^2}{(1-\bar{a} \, e^{i\theta})(e^{i\theta}-a)} = \frac{1-|a|^2}{|e^{i\theta}-a|^2} \,\, e^{-i\theta}.
\]
Therefore, $\T_{-}$ is mapped to an arc in $\C_+$. Moreover, on the interval $t \in (-1,1)$, we have
\[
\Phi(t) = \frac{1-|a|^2}{(1-\bar{a} \, t)(t-a)} = \frac{1-|a|^2}{|t-a|^2}  \times \frac{t-\bar{a}}{1-\bar{a} \, t} = \frac{1-|a|^2}{|t-a|^2} \,\, T_t(\bar{a}).
\]
Thus $(-1,1)$ is also mapped to an arc in $\C_+=\{z\in\C:\Im z>0\}$. In short, $\partial \D_{-}$ is mapped to a closed arc in $\C_+$.
Therefore, we deduce that $\Phi$ maps $\D_{-}$ into $\C_+$. Note that $0$ is not contained in $\Phi(\D_{-})$.
This fact implies that $B'$ does not have any zeros in $\D_{-}$.

By continuity, we can say that if all zeros of $B$ are in $\overline{\D}_{+}$, then so does all the zeros of $B'$ which are in the disc.

A general automorphism of $\D$ has the form $T_{a,\gamma}=\gamma T_a$, for some
$\gamma \in \T$ and $a \in \D$. Let $f = B \circ \, T_{a, \gamma}$. Then $f$ is also a finite Blaschke product with zeros $T_a(\bar{\gamma} \, z_1)$, $T_a(\bar{\gamma} \, z_2)$,  $\dots, \, T_a(\bar{\gamma} \, z_n)$. Moreover, if we denote the zeros of $B'$ in $\D$ by $w_1,w_2,\dots,w_{n-1}$, the zeros of $f'$ in $\D$ would be $T_a(\bar{\gamma} \, w_1), \, T_a(\bar{\gamma} \, w_2),\dots,T_a(\bar{\gamma} \, w_{n-1})$. Considering the boundary curves of the hyperbolic convex hull of $z_1,\dots,z_n$, if we choose $a$ and $\gamma$ such that
\[
T_a(\bar{\gamma} \, z_1), \, T_a(\bar{\gamma} \, z_2),  \dots, \, T_a(\bar{\gamma} \, z_n) \in \overline{\D}_{+},
\]
then the preceding observation shows that
\[
T_a(\bar{\gamma} \, w_1), \, T_a(\bar{\gamma} \, w_2),\dots,T_a(\bar{\gamma} \, w_{n-1}) \in \overline{\D}_{+}.
\]
Therefore, we see that if the zeros of $B$ are on one side of the hyperbolic line
\[
\gamma \, \frac{a-z}{1-\bar{a} \, z} = t, \hspace{1cm} t \in [-1,1],
\]
then the zeros of $B'$ are also on the same side. The intersection of all such lines gives the hyperbolic convex hull of the zeros of $B$.
\end{proof}

\noindent {\em Remark}: The argument above also works for infinite Blaschke products.

Theorem \ref{T:hyper-convex} is sharp in the sense which is crystalized by the following examples. Let $a,b \in \D$ and put
\[
B(z) = \left( \frac{a-z}{1-\bar{a} \, z} \right)^m \left( \frac{b-z}{1-\bar{b} \, z} \right)^n.
\]
Clearly $B'$ has $m+n-1$ zeros in $\D$ which are $a$, $m-1$ times, and $b$, $n-1$ times, and the last one $c$ is given by the equation
\[
\frac{m(1-|a|^2)}{(1-\bar{a} \, z)^2} \left( \frac{a-z}{1-\bar{a} \, z} \right)^{m-1}  \left( \frac{b-z}{1-\bar{b} \, z} \right)^n + \left( \frac{a-z}{1-\bar{a} \, z} \right)^m  \left( \frac{b-z}{1-\bar{b} \, z} \right)^{n-1} \frac{n(1-|b|^2)}{(1-\bar{b} \, z)^2} = 0.
\]
Rewrite this equation as
\[
\left( \frac{z-a}{1-a \, \bar{z}} \right) / \left(\frac{z-b}{1-b \, \bar{z}} \right) =- \left(\frac{m(1-|a|^2)}{|1-\bar{a} \, z|^2} \right) / \left(\frac{n(1-|b|^2)}{|1-\bar{b} \, z|^2}\right) ,
\]
which, by (\ref{E:z1z2z3online}), reveals that $a,b,c$ are on the same hyperbolic line. Moreover, as $m$ and $n$ are any arbitrary positive integers, the point $c$ traverses a dense subset of the hyperbolic line segment between $a$ and $b$.

With more delicate calculations, we can consider examples of the form
\[
B(z) = \left( \frac{a-z}{1-\bar{a} \, z} \right)^m \, \, \left( \frac{b-z}{1-\bar{b} \, z} \right)^n\, \, \left( \frac{c-z}{1-\bar{c} \, z} \right)^p,
\]
where $a,b,c \in \D$ and $m,n,p \geq 1$, and observe that the zeros of $B'$, for different values of $m,n,p$, form a dense subset of the hyperbolic convex hull of $a,b,c$.

\section{Uniform convergence on compact sets}

The following result is not stated in its general form. We are content with this version which is enough to obtain our main result that comes afterward.

\begin{Lem} \label{L:unif-away-0}
Let
\[
B(z) = \gamma \, \prod_{k=1}^{n} \frac{z_k-z}{1-\bar{z}_k \, z}
\]
and fix any $M$ such that
\[
\max\{\, |z_1|, \, |z_2|, \, \dots, \, |z_n| \} < M < 1.
\]
Then there is a constant $\delta = \delta(M,B)>0$ such that, for any two distinct points $z,w$ in the annulus $\{z : M \leq |z| \leq 1/M\}$, we have
\[
|z-w|<\delta \Longrightarrow B(z) \ne B(w).
\]
\end{Lem}

\begin{proof}
Suppose this is not true. Then, for each $k \geq 1$, there are $a_k$ and $b_k$ such that
\[
M \leq |a_k|, \, |b_k| \leq 1/M, \,\,\, a_k \ne b_k, \,\,\,  |a_k-b_k|<1/k,
\]
but $B(a_k) = B(b_k)$. Since the closed annulus is compact, $(a_k)_{k \geq 1}$ has a convergent subsequence. Without loss of generality, we may assume that $a_k \longrightarrow a$ for some $a$ with $M \leq |a| \leq 1/M$. This implies $b_k \longrightarrow a$. Therefore, we would have
\[
B'(a) = \lim_{k \to \infty} \frac{B(a_k)-B(b_k)}{a_k-b_k} = 0.
\]
This is a contradiction. as a matter of fact, by Theorem \ref{T:hyper-convex}, the zeros of $B'$ in $\D$ are in the hyperbolic convex hull of the zeros of $B$, and those outside $\D$ are the reflection of the zeros inside with respect to $\T$. In other words, $B'$ has no zeros on the annulus.
\end{proof}

There is another way to arrive at a weaker version of Lemma \ref{L:unif-away-0}, which will also suffice for us. On $\T$, the nice formula
\begin{equation}\label{eq:derive-au-bord}
|B'(e^{i\theta})| = \sum_{k=1}^{n} \frac{1-|z_k|^2}{|e^{i\theta}-z_k|^2}
\end{equation}
holds. Hence, at least, $B'(e^{i\theta}) \ne 0$ for all $e^{i\theta} \in \T$. Thus, by continuity, there is an annulus on which $B'$ has no zeros, and as we explained above, for such a region $\delta$ exists.

For the sake of completeness, we remind that if $a_k \in \D$, and $\gamma_k,\gamma_0 \in \T$, are such that
\[
a_k \gamma_k \longrightarrow \gamma_0, \qquad \hbox{as } k \longrightarrow +\infty,
\]
then the simple estimation
\[
\left| \gamma_0 - \gamma_k \,  \frac{a_k-z}{1-\bar{a}_k \, z} \right| \leq \frac{2\,
|\gamma_0-a_k \gamma_k|}{1-|z|}, \qquad \hbox{as } k \longrightarrow +\infty,
\]
reveals that
\begin{equation} \label{E:sakgakroi}
T_{a_k,\gamma_k}(z) \longrightarrow \gamma_0
\end{equation}
uniformly on compact subsets of $\D$.

\begin{Thm} \label{T:conv-sak}
Let $B$ be any finite Blaschke product of order $\geq 1$. Let $a_k \in \D$ and $\gamma_k,\gamma_0 \in \T$, be such that
\[
\lim_{k \to \infty} a_k \gamma_k = \gamma_0.
\]
Then
\[
T_{B(a_k\gamma_k),\, \bar{\gamma}_k} \circ B \circ T_{a_k,\gamma_k} \,\,
\longrightarrow \,\, \rho_\gamma,
\]
where $\gamma = B'(\gamma_0)/|B'(\gamma_0)|$. In other words,
\[
\big(\, T_{B(a_k\gamma_k),\, \bar{\gamma}_k} \circ B \circ T_{a_k,\gamma_k} \,\big) (z) \,\,
\longrightarrow \,\, \frac{B'(\gamma_0)}{|B'(\gamma_0)|} \, z
\]
as $k \longrightarrow +\infty$ and the convergence is uniform on compact subsets of $\D$.
\end{Thm}

\begin{proof}
Suppose that $B$ has order $n$. Then, for each fixed $k$, the function
\begin{equation} \label{E:formul-fk}
f_k = T_{B(a_k\gamma_k),\, \bar{\gamma}_k} \circ B \circ T_{a_k,\gamma_k}
\end{equation}
is also a finite Blaschke product of order $n$. Its zeros, say
\[
z_{k,1}, \,\,\,
z_{k,2}, \,\,\,
\dots, \,\,\,
z_{k,n}, \,\,\,
\]
are the solutions of the equation
\[
B \big(\, T_{a_k,\gamma_k}  (z) \,\big) = B(a_k\gamma_k).
\]
Certainly $z=0$ is a solution. We need to show that the other zeros accumulate to some points on $\T$.

Denote the solutions of $B(w) = B(a_k\gamma_k)$ by
\[
w_{k,1}, \,\,\, w_{k,2}, \,\,\, \dots, \,\,\, w_{k,n},
\]
and put $w_{k,1} = a_k\gamma_k$. As $k \longrightarrow \infty$, we have
\[
|w_{k,j}| \longrightarrow 1, \hspace{1cm} (1 \leq j \leq n),
\]
and, by Lemma \ref{L:unif-away-0},
\[
|w_{k,i}-w_{k,j}| \geq \delta, \hspace{1cm} (1 \leq i \ne j \leq n).
\]
By the transformation $z=T_{a_k,\gamma_k}^{-1}(w)$, the first solution is mapped to the origin. But the other solutions, thanks to the positive distance between them, are mapped to points close to $\T$. More precisely, we have
\begin{equation} \label{E:zkjgoest}
\lim_{k \to \infty} |z_{k,j}| = \lim_{k \to \infty} |T_{a_k,\gamma_k}^{-1}(w_{k,j})| = 1, \hspace{1cm} (2 \leq j \leq n).
\end{equation}

Write
\[
f_k(z) = \eta_k \, z \, \prod_{j=2}^{n} \frac{z_{k,j}-z}{1-\bar{z}_{k,j} \, z},
\]
where $\eta_k \in \T$. This formula shows
\[
f_k'(0) = \eta_k \, \prod_{j=2}^{n} z_{k,j}
\]
and (\ref{E:formul-fk}) reveals that
\begin{eqnarray*}
f_k'(0) &=& T'_{B(a_k\gamma_k),\, \bar{\gamma}_k} (B(a_k\gamma_k)) \,\,\, B'(a_k\gamma_k) \,\,\, T'_{a_k,\gamma_k}(0) \\
&=&  \frac{1-|a_k\gamma_k|^2}{1-|B(a_k\gamma_k)|^2} \,\,\, B'(a_k\gamma_k).
\end{eqnarray*}
Hence, we obtain the representation
\[
f_k(z) = \frac{1-|a_k\gamma_k|^2}{1-|B(a_k\gamma_k)|^2} \,\,\, B'(a_k\gamma_k) \,\, \left( \prod_{j=2}^{n} \frac{1}{|z_{k,j}|} \times \prod_{j=2}^{n} \frac{|z_{k,j}|}{z_{k,j}} \,\, \frac{z_{k,j}-z}{1-\bar{z}_{k,j} \, z}\right) \,\, z.
\]
By (\ref{E:sakgakroi}) and (\ref{E:zkjgoest}),
\[
\left( \prod_{j=2}^{n} \frac{1}{|z_{k,j}|} \times \prod_{j=2}^{n} \frac{|z_{k,j}|}{z_{k,j}} \,\,
\frac{z_{k,j}-z}{1-\bar{z}_{k,j} \, z}\right) \longrightarrow 1,
\]
as $k \to \infty$. It is also known that
\[
\lim_{z \to \gamma_0} \frac{1-|B(z)|^2}{1-|z|^2} = |B'(\gamma_0)|.
\]
We are done.
\end{proof}

\begin{Cor} \label{C:conv-sak}
Let $B$ be a finite Blaschke product. Let $a_k \in \D$ and $\gamma_0 \in \T$ be such that
\[
\lim_{k \to \infty} a_k  = \gamma_0.
\]
Then
\[
T_{B(a_k)} \circ B \circ T_{a_k} \,\,
\longrightarrow \,\, \rho_\gamma, \qquad (k \longrightarrow +\infty),
\]
where $\gamma = B'(\gamma_0)/|B'(\gamma_0)|$.
\end{Cor}

A simple example shows that in Theorem \ref{T:conv-sak}, $T_{B(a_k\gamma_k),\, \bar{\gamma}_k}$ cannot be replaced by $T_{B(a_k\gamma_k)}$. In other words, the rotation by $\bar{\gamma}_k$ is needed to obtain the convergence. To see this fact, let $a_k$ be any positive sequence on $[0,1)$ tending to $1$. Let $\gamma_k = (-1)^k$, and put $B(z)=z^2$. Then
\[
\big(\, T_{B(a_{2k}\gamma_{2k})} \circ B \circ T_{a_{2k},\gamma_{2k}} \,\big) (z) \,\, \longrightarrow  \, z
\]
while
\[
\big(\, T_{B(a_{2k+1}\gamma_{2k+1})} \circ B \circ T_{a_{2k+1},\gamma_{2k+1}} \,\big) (z) \,\, \longrightarrow  \, -z.
\]
Hence, $T_{B(a_k\gamma_k)} \circ B \circ T_{a_k,\gamma_k}$ is not convergent. Of course, either by Theorem \ref{T:conv-sak} or by direct verification,
\[
\big(\, T_{B(a_k\gamma_k),\, \bar{\gamma}_k} \circ B \circ T_{a_k,\gamma_k} \,\big) (z) \,\, \longrightarrow  \, z,
\]

\end{document}